\documentclass[a4paper,reqno]{amsart}

\usepackage[utf8]{inputenc}
\usepackage[english]{babel}
\usepackage{microtype} 

\usepackage{hyperref} 
\hypersetup{hidelinks}
\usepackage[nameinlink]{cleveref} 

\usepackage{amsmath,amssymb,amsthm}
\usepackage{thmtools} 
\usepackage{mathtools} 
\usepackage{upgreek}
\usepackage{accents}

\usepackage{tikz} 
\usetikzlibrary{cd}
\usetikzlibrary{calc}

\usepackage{orcidlink}

\usepackage{multirow}

\usepackage{dsfont} 
\usepackage{pifont} 
\usepackage{nicefrac} 

\usepackage{ifthen}
\usepackage{enumitem}

\newcounter{i} 
\newtoks\striche 

\newcommand{\smatvdots}{\vphantom{\int\limits^x}\smash{\vdots}} 

\newcommand{\R}{\mathbb{R}}
\newcommand{\N}{\mathbb{N}} 

\newcommand{\C}{\mathbb{C}}

\newcommand{\Lp}[1]{\mathrm{L}^{#1}} 

\newcommand{\diffd}{\mathrm{d}} 
\newcommand{\dx}[1][x]{\,\diffd#1}

\newcommand{\cl}[2][]{\overline{#2}\ifthenelse{ \equal{#1}{} }{}{^{#1}}} 

\newcommand{\ball}{\mathrm{B}}

\DeclareMathOperator{\ran}{ran}

\DeclareMathOperator{\spn}{span}

\DeclareMathOperator{\supp}{supp}

\DeclareMathOperator{\Div}{div}
\renewcommand{\div}{\Div}
\newcommand{\grad}{\nabla}
\newcommand{\tangrad}{\nabla_{\tau}}
\newcommand{\wtangrad}{\widetilde{\nabla}_{\tau}}
\DeclareMathOperator{\rot}{curl}

\DeclarePairedDelimiter{\set}{\{}{\}}
\DeclarePairedDelimiter{\norm}{\lVert}{\rVert}
\DeclarePairedDelimiter{\abs}{\lvert}{\rvert}

\DeclarePairedDelimiterX{\dset}[2]{\{}{\}}{#1\,\delimsize\vert\,\mathopen{} #2}
\DeclarePairedDelimiterX{\scprod}[2]{\langle}{\rangle}{#1,#2}
\DeclarePairedDelimiterX{\dualprod}[2]{\langle}{\rangle}{#1,#2}
\DeclarePairedDelimiterX{\sdprod}[2]{\llangle}{\rrangle}{#1,#2} 

\newcommand{\adjunsymb}{\ast} 

\newcommand{\adjun}[1][1]{%
  \setcounter{i}{1}%
  \striche={\adjunsymb}%
  \loop%
  \ifnum\value{i}<#1%
  \striche=\expandafter{\the\expandafter\striche\adjunsymb}%
  \stepcounter{i}%
  \repeat%
  ^{\the\striche}%
}

\newcommand{\mapping}[4]{%
  \left\{%
    \begin{array}{rcl}%
      #1 &\to & #2, \\
      #3 &\mapsto & #4
    \end{array}%
  \right.%
}

\newcommand{\transposed}{^{\mathsf{T}}}
\newcommand{\trans}{\transposed}

\newcommand{\Hspace}{\mathrm{H}}

\newcommand{\cH}{\accentset{\circ}{\Hspace}}

\newcommand{\conC}{\mathrm{C}}
\newcommand{\Cc}[1][\infty]{\accentset{\circ}{\mathrm{C}}^{#1}}

\newcommand{\boundtr}[1][]{\gamma_{0}\ifthenelse{\equal{#1}{}}{}{^{#1}}}
\newcommand{\normaltr}[1][]{\gamma_{\nu}\ifthenelse{\equal{#1}{}}{}{^{#1}}}
\newcommand{\tantr}[1][]{\pi_{\tau}\ifthenelse{\equal{#1}{}}{}{^{#1}}}
\newcommand{\tanxtr}[1][]{\gamma_{\tau}\ifthenelse{\equal{#1}{}}{}{^{#1}}}

\theoremstyle{plain}
\newtheorem{theorem}{Theorem}[section]
\newtheorem{lemma}[theorem]{Lemma}
\newtheorem{proposition}[theorem]{Proposition}
\newtheorem{corollary}[theorem]{Corollary}

\theoremstyle{definition}
\newtheorem{definition}[theorem]{Definition}

\newtheorem{claim}{Claim}

\theoremstyle{remark}

\title[Sobolev spaces on a strong Lipschitz boundary]{Characterizations of the Sobolev space $\Hspace^{1}$ on the boundary of a strong Lipschitz domain in 3-D}

\author[Nathanael Skrepek]{Nathanael Skrepek\,\orcidlink{0000-0002-3096-4818}}
\thanks{\textit{E-mail}: \href{mailto:nathanael.skrepek@math.tu-freiberg.de}{nathanael.skrepek@math.tu-freiberg.de}}
\date{\today}

\address{TU Bergakademie Freiberg \\
  Institute of Applied Analysis \\
  Akademiestraße 6 \\
  D-09596 Freiberg \\
  Germany}
\email{nathanael.skrepek@math.tu-freiberg.de}

\begin{document}

\begin{abstract}
  In this work we investigate the Sobolev space $\Hspace^{1}(\partial\Omega)$ on a strong Lipschitz boundary $\partial\Omega$, i.e., $\Omega$ is a strong Lipschitz domain. In most of the literature this space is defined via charts and Sobolev spaces on flat domains.
  We show that there is a different approach via differential operators on $\Omega$ and a weak formulation directly on the boundary that leads to the same space.
  This second characterization of $\Hspace^{1}(\partial\Omega)$ is in particular of advantage, when it comes to traces of $\Hspace(\rot,\Omega)$ vector fields.
\end{abstract}

\maketitle

\section{Introduction}

We will give two characterizations of $\Hspace^{1}(\partial\Omega)$, where $\Omega$ is a strong Lipschitz domain. The first is given via charts, which is the usual approach in literature, and the second is a weak characterization directly on the boundary, which is related to the weak characterization of an $\Lp{2}(\partial\Omega)$ tangential trace for $\Hspace(\rot,\Omega)$ fields.

Our main motivation is that the result we present serves us to fill details in \cite[Proof of Thm.~2]{Cos90}, \cite[Section Le cas tridimensionnel]{BeBeCoDa97}, \cite[Proof of Thm.~5.1]{BuCoSh02} and \cite[Proof of Lem.~3.53]{monk}, where it is used. Unfortunately, without an explanation or a reference for its validity. Hence, we decided to address this issue.

In particular, if we regard an $f \in \Hspace^{1}(\Omega)$, then $\grad f \in \Hspace(\rot,\Omega)$ follows automatically. Every element of $\Hspace(\rot,\Omega)$ possesses a tangential trace in an abstract boundary space and therefore also $\grad f$ possesses a tangential trace.
For smooth functions the tangential trace is well defined as an element of $\Lp{2}(\partial\Omega)^{3}$. Moreover, for a smooth function the tangential trace of its gradient field coincides with the boundary gradient of its restriction to the boundary, see \Cref{le:characterization-of-tangrad}.
This suggests the following claim.
\begin{claim}\label{cl:motivation}
  Let $f \in \Hspace^{1}(\Omega)$. If the tangential trace of $\grad f$ belongs to $\Lp{2}(\partial\Omega)^{3}$, then $f \big\vert_{\partial\Omega}$ belongs to $\Hspace^{1}(\partial\Omega)$.
\end{claim}
However, there are two approaches to define ``\emph{the tangential trace belongs to $\Lp{2}(\partial\Omega)^{3}$}'': The strong approach via limits of smooth functions and the weak approach via a representation by an $\Lp{2}(\partial\Omega)$ inner product.
For the strong approach it is not hard to show that \Cref{cl:motivation} is true.
However, it is more relevant to answer the question for the weak approach. Hence, we regard the claim with the weak characterization of $\Lp{2}$ tangential traces.

In fact both \cite{BeBeCoDa97} and \cite{monk} are using \Cref{cl:motivation} (with weak $\Lp{2}$ tangential traces) to prove that both approaches (strong and weak) to $\Lp{2}$ tangential traces lead to the same objects, i.e., weak = strong.
Hence, in order to avoid a circular argument we have to resist the temptation to prove \Cref{cl:motivation} for strong $\Lp{2}$ tangential traces and conclude it for weak by ``weak = strong''.


In order to avoid the introduction of unnecessarily many concepts, we broke down the question to its core, which is an alternative approach to $\Hspace^{1}(\partial\Omega)$, see \Cref{def:weak-H1-on-boundary}. Hence, we do not need the space $\Hspace(\rot,\Omega)$ and the abstract tangential trace at all, although these notions are the origin of the question. Nevertheless, in \Cref{sec:back-to-original-question} we come back to the original question and show that \Cref{cl:motivation} holds true.

\section{Strong Lipschitz boundaries}

Recall the definition of a strong Lipschitz domain, see, e.g., \cite{grisvard1985}.

\begin{definition}\label{def:strong-lipschitz-domain}
  Let $\Omega$ be an open subset of $\R^{d}$. We say $\Omega$ is a \emph{strong Lipschitz domain}, if for every $p \in \partial\Omega$ there exist $\epsilon,h > 0$, a hyperplane $W = \spn\set{w_{1},\dots,w_{d-1}}$, where $\set{w_{1},\dots,w_{d-1}}$ is an orthonormal basis of $W$, and a Lipschitz continuous function $a\colon (p + W) \cap \ball_{\epsilon}(p) \to (-\frac{h}{2},\frac{h}{2})$ such that
  \begin{align*}
    \partial\Omega \cap C_{\epsilon,h}(p) &= \dset{x + a(x)v}{x \in (p+W) \cap \ball_{\epsilon}(p)}, \\
    \Omega \cap C_{\epsilon,h}(p) &= \dset{x + sv}{x \in (p+W) \cap \ball_{\epsilon}(p), -h < s < a(x)},
  \end{align*}
  where $v$ is the normal vector of $W$ and $C_{\epsilon,h}(p)$ is the cylinder $\dset{x + \delta v}{x \in (p+ W) \cap \ball_{\epsilon}(p), \delta \in (-h,h)}$.

  The boundary $\partial\Omega$ is then called \emph{strong Lipschitz boundary}.
\end{definition}

Note that the condition $\abs{a} < \frac{h}{2}$ is not really necessary, however it reduces technical constructions. If it was not already satisfied, we can force it by shrinking $\epsilon$.

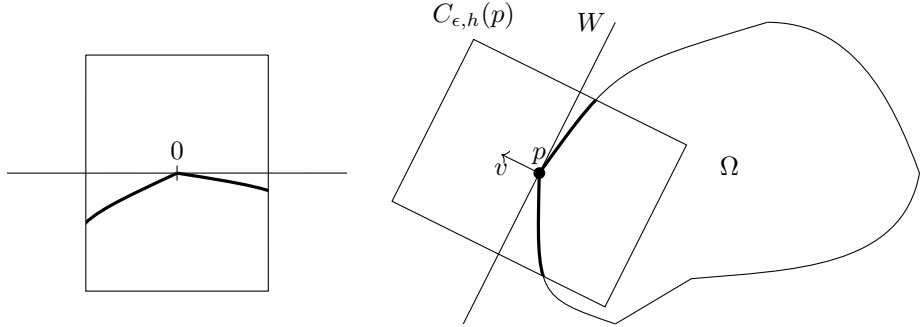
\begin{figure}
  \newcommand{\nvec}{v}
  \centering
  \begin{tikzpicture}[scale=2]
    \coordinate (vec-n) at (-1,0.5); 
    \coordinate (vec-t) at (0.5,1); 

    \coordinate[label=above:$p$] (p) at (0.5,1);
    \coordinate[label=below:$\nvec$] (n) at ($(p) + 0.25*(vec-n)$);

    \coordinate (H1) at ($(p) - (vec-t)$);
    \coordinate[label=left:$W$] (H2) at ($(p) + (vec-t)$);

    \coordinate (C1) at ($(p) + 0.6*0.89442719*(vec-t) + 0.7*(vec-n)$);
    \coordinate (C2) at ($(p) + 0.6*0.89442719*(vec-t) - 0.7*(vec-n)$);
    \coordinate (C3) at ($(p) - 0.6*0.89442719*(vec-t) - 0.7*(vec-n)$);
    \coordinate (C4) at ($(p) - 0.6*0.89442719*(vec-t) + 0.7*(vec-n)$);


    \draw (H1) -- (H2);

    \tikzstyle{transformation} = [shift={(-3,1)},rotate=-63.43]

    \draw ([transformation]p) + (0,0.05) -- ++(0,-0.05);
    \node[shift={(0,0.3)}] at ([transformation]p) {$0$};
    \draw ([transformation]C1) -- ([transformation]C2) -- ([transformation]C3) -- ([transformation]C4) -- ([transformation]C1);
    \draw ([transformation]H1) -- ([transformation]H2);
    \begin{scope}
      \clip ([transformation]C1) -- ([transformation]C2) -- ([transformation]C3) -- ([transformation]C4) -- ([transformation]C1);
      \draw[transformation,very thick] (1,0) .. controls (0.48,0.18) .. ([transformation]p);
      \draw[transformation,very thick] ([transformation]p) .. controls (1,1.7) .. (2,2);
    \end{scope}

    \draw (C1) -- (C2) -- (C3) -- (C4) -- (C1);
    \node[above] at (C1) {$C_{\epsilon,h}(p)$};

    \draw[->] (p) -- (n); 

    \draw (1,0) .. controls (0.48,0.18) .. (p);
    \draw (p) .. controls (1,1.7) .. (2,2);
    \draw
    (2,2) to [out=0,in=110] (3,1)
    (3,1) to [out=-100,in=5] (1.5,0.3)
    (1.5,0.3) to (1,0);

    \begin{scope}
      \clip (C1) -- (C2) -- (C3) -- (C4) -- (C1);
      \draw[very thick] (1,0) .. controls (0.48,0.18) .. (p);
      \draw[very thick] (p) .. controls (1,1.7) .. (2,2);
    \end{scope}

    \node at (1.75,1.05) {$\Omega$};
    \filldraw (p) circle (0.1em); 
  \end{tikzpicture}
  \caption{Lipschitz boundary}
  \label{fig:lipschitz-boundary}
\end{figure}

Locally the boundary is given by the graph of a Lipschitz function, see \Cref{fig:lipschitz-boundary}. Therefore, we can define Lipschitz charts on $\partial\Omega$ in the following way. Let $p$, $C_{\epsilon,h}(p)$, $W$, $v$, $a$ be as in \Cref{def:strong-lipschitz-domain}. We will also denote the matrix that contains the orthonormal basis of $W$ as columns by $W$, i.e., $W \in \R^{d \times (d-1)}$. Hence, the mapping $\zeta \mapsto W\trans \zeta$ gives the coordinates (w.r.t.\ the basis $w_{1},\dots,w_{d-1}$) of the orthogonal projection of $\zeta$ on the hyperplane $W$.
We introduce a \emph{strong Lipschitz chart} locally at $p$ by
\begin{align*}
  k\colon
  \mapping{\partial\Omega \cap C_{\epsilon,h}(p)}{\ball_{\epsilon}(0) \subseteq \R^{d-1}}{\zeta}{W\trans(\zeta-p).}
\end{align*}
We say that $\Gamma \coloneqq \partial\Omega \cap C_{\epsilon,h}(p)$ is the \emph{chart domain} of $k$. Also every restriction of a chart to an open non-empty $\hat{\Gamma} \subseteq \Gamma$ (w.r.t.\ the trace topology) is again a chart with chart domain $\hat{\Gamma}$. The corresponding inverse chart is given by
\begin{align*}
  k^{-1} \colon
  \mapping{\ball_{\epsilon}(0) \subseteq \R^{d-1}}{\partial\Omega \cap C_{\epsilon,h}(p)}{x}{p + \sum_{i=1}^{d-1} x_{i} w_{i} + a(p + \sum_{i=1}^{d-1} x_{i} w_{i}) v.}
\end{align*}
In the case where $k$ is a ``restricted'' chart, we have $k^{-1} \colon U \to \hat{\Gamma}$, where $U$ is an open non-empty subset of $\ball_{\epsilon}(0)$ in $\R^{d-1}$.
For notational simplicity we just write $a(x)$ instead of $a(p +\sum_{i=1}^{d-1} x_{i} w_{i})$. By this convention we have $a \colon U \subseteq \R^{d-1} \to \R$.


Note that in fact $W$, $v$ and $p$ establish an alternative coordinate system with origin $p$. Hence, by translation and rotation we can, most of the time, assume (w.l.o.g.) that $W = (e_{1},\dots, e_{d-1})$, $v = e_{d}$ and $p = 0$. This will also better transport the essence of our ideas. In this coordinate system we have
\begin{align*}
  k\left(
  \begin{bmatrix}
    \zeta_{1} \\ \vdots \\ \zeta_{d}
  \end{bmatrix}
  \right) =
  \begin{bmatrix}
    \zeta_{1} \\ \vdots \\ \zeta_{d-1}
  \end{bmatrix}
  \quad\text{and}\quad
  k^{-1}(x) =
  \begin{bmatrix}
    x \\ a(x)
  \end{bmatrix}.
\end{align*}
However, sometimes it is not entirely obvious that we can reduce the general setting to this situation or the justification that such a reduction is valid is as difficult as working in the general setting in the first place. Hence, for completeness we will repeat the tricky parts for the general setting in the appendix.

\medskip

Let $k \colon \Gamma \to U$ be a strong Lipschitz chart. The surface measure on $\partial\Omega$ is locally given by
\begin{align*}
  \mu(\Upsilon) = \int_{k(\Upsilon)} \sqrt{\det (\diffd k^{-1})\trans \diffd k^{-1}} \dx[\uplambda_{d-1}]
  \quad\text{for}\quad \Upsilon \subseteq \Gamma,
\end{align*}
where $\uplambda_{d-1}$ is the Lebesgue measure in $\R^{d-1}$. The surface measure is then defined by a partition of $\partial\Omega$. The surface measure is independent of the partition and the charts, see \Cref{th:surface-measure-independent-of-charts}.
Hence, we can switch between the inner products of $\Lp{2}(\Gamma)$ and $\Lp{2}(U)$ by
\begin{align*}
  \scprod{f}{g}_{\Lp{2}(\Gamma)} = \scprod[\Big]{f \circ k^{-1}}{\sqrt{\det (\diffd k^{-1})\trans \diffd k^{-1}}\, g \circ k^{-1}}_{\Lp{2}(U)}.
\end{align*}

\section{Preparation and main result}

We will use for spaces with homogeneous boundary conditions the same notation as in \cite{BaPaScho16}: For an open set $M \subseteq \R^{d}$ we denote the set of $\conC^{\infty}$ functions with compact support in $M$ by
\begin{equation*}
  \Cc(M) \coloneqq \dset{\Phi \in \conC^{\infty}(\R^{d})}{\supp \Phi \subseteq M \;\text{and}\; \supp \Phi \;\text{is compact}}.
\end{equation*}
Moreover, we denote the standard $\Lp{2}(M)$ first order Sobolev space by $\Hspace^{1}(M)$ and
\begin{equation*}
  \cH^{1}(M) \coloneqq \mathrlap{\cl[\Hspace^{1}(M)]{\phantom{\Cc(M)}}} \mathrlap{\Cc(M)} \phantom{\cl[\Hspace^{1}(M)]{\Cc(M)}}.
\end{equation*}
The circle on top of $\Hspace^{1}(M)$ indicates homogeneous boundary conditions.

\medskip

In the following we assume $\Omega \subseteq \R^{3}$ to be a strong Lipschitz domain. Moreover, we will assume that the strong Lipschitz charts $k\colon \Gamma \subseteq \partial \Omega \to U$ are of the following form
\begin{equation*}
  k^{-1} \colon \mapping{U}{\Gamma}{%
    \begin{bmatrix}
      x_{1} \\ x_{2}
    \end{bmatrix}
  }{\begin{bmatrix}
      x_{1} \\ x_{2} \\ a(x_{1},x_{2}),
    \end{bmatrix}}
\end{equation*}
where $U$ is an open subset of $\R^{2}$ and $a \colon U \to \R$ is a Lipschitz continuous mapping.
The normal vector is then given by
\begin{align*}
  \nu(k^{-1}(x_{1},x_{2})) = \frac{1}{\sqrt{1 + \norm{\grad a}^{2}}}
  \begin{bmatrix}
    - \partial_{1} a(x_{1},x_{2}) \\
    - \partial_{2} a(x_{1},x_{2}) \\
    1
  \end{bmatrix}.
\end{align*}

In \Cref{sec:general-charts} we show, which modifications have to be done when we work with ``general'' strong Lipschitz charts. We could also do everything for ``general'' strong Lipschitz charts in the first place, however it does not transport the underlying ideas that well. Also we did not want to just say that we can always reduce every thing to these ``special'' strong Lipschitz charts, as sometimes it is not obvious how this ``w.l.o.g.'' is justified.

\begin{lemma}\label{le:functional-determinant-of-chart}
  Let $k$ be a strong Lipschitz chart. Then
  \begin{align*}
    \det\big((\diffd k^{-1})\trans \diffd k^{-1}\big) = 1 + \norm{\grad a}^{2}.
  \end{align*}
\end{lemma}

\begin{proof}
  Note that
  \begin{align*}
    \diffd k^{-1} =
    \begin{bmatrix}
      1 & 0 \\
      0 & 1 \\
      \partial_{1} a & \partial_{2} a
    \end{bmatrix}
    \quad\text{and}\quad
    (\diffd k^{-1})\trans \diffd k^{-1}
    =
    \begin{bmatrix}
      1 & 0 \\
      0 & 1
    \end{bmatrix}
    +
    \begin{bmatrix}
      \partial_{1} a \\ \partial_{2} a
    \end{bmatrix}
    \begin{bmatrix}
      \partial_{1} a & \partial_{2} a
    \end{bmatrix}.
  \end{align*}
  Hence, \Cref{le:determinant-of-I-plus-vvT} implies the claim.
\end{proof}

Recall the Moore-Penrose inverse: For an injective matrix $A$ it is given by $A^{\dagger} = (A\trans A)^{-1} A\trans$. Our first approach to the first order Sobolev space on $\partial\Omega$ is well-known, see, e.g., \cite[beginning of Sec.~3]{BuCoSh02}, \cite[Def.~1.3.3.2]{grisvard1985} or \cite[after Thm.~4.10]{necas2012}.

\begin{definition}\label{def:strong-H1-on-boundary}
  Let $f \in \Lp{2}(\partial\Omega)$. We say $f \in \Hspace^{1}(\partial\Omega)$, if for every strong Lipschitz chart $k \colon \Gamma \to U$ we have $f \circ k^{-1} \in \Hspace^{1}(U)$. The tangential gradient is then defined by
  \[
    (\tangrad f) \big\vert_{\Gamma} = \big[\diffd (f \circ k^{-1}) (\diffd k^{-1})^{\dagger}\big]\trans \circ k = \big[{(\diffd k^{-1})^{\dagger}}\trans \grad_{\R^{2}}(f \circ k^{-1}) \big]\circ k.
  \]
  We endow $\Hspace^{1}(\partial\Omega)$ with the following norm
  \[
    \norm{f}_{\Hspace^{1}(\partial\Omega)} = \sqrt{\norm{f}_{\Lp{2}(\partial\Omega)}^{2} + \norm{\tangrad f}_{\Lp{2}(\partial\Omega)}^{2}}.
  \]
\end{definition}

Note that, if the previous definition is true for a set of charts whose chart domains cover $\partial\Omega$, then it is already true for all charts. Moreover, the definition of the tangential gradient is independent of the chart, see \Cref{th:tangential-derivative-independent-of-charts}.

\medskip
Note that for a.e.\ $\zeta \in \partial\Omega$ the tangential space is spanned by the columns of $\diffd k^{-1}(k(\zeta))$. We denote the space of all $\Lp{2}(\partial\Omega)$ vector fields that are point wise a.e.\ in the tangential space by
\begin{align*}
  \Lp{2}_{\tau}(\partial\Omega) \coloneqq \dset{g \in \Lp{2}(\partial\Omega)^{3}}{\nu \cdot g = 0}.
\end{align*}
By construction $\tangrad f$ belongs to $\Lp{2}_{\tau}(\partial\Omega)$. This can be seen by
\begin{equation*}
  (\nu \cdot \tangrad f) \circ k^{-1} = \nu \circ k^{-1} \cdot \diffd k^{-1} \big((\diffd k^{-1})\trans \diffd k^{-1}\big)^{-1} \grad_{\R^{2}}(f \circ k^{-1}) = 0,
\end{equation*}
because $\nu \circ k^{-1} \perp \diffd k^{-1}$ by definition.

The orthogonal projection on $\Lp{2}_{\tau}(\partial\Omega)$ is given by $q \mapsto (\nu \times q) \times \nu$. For a $Q \in \Cc(\R^{3})^{3}$ we define the \emph{tangential trace} by
\begin{align*}
  \tantr Q \coloneqq \big(\nu \times Q \big\vert_{\partial\Omega}\big) \times \nu
\end{align*}
For smooth functions $F \in \Cc(\R^{3})$ the next lemma shows that the tangential gradient on $\partial\Omega$ matches the tangential trace of the volume gradient on $\Omega$.

\begin{lemma}\label{le:characterization-of-tangrad}
  For $F \in \Cc(\R^{3})$ we have $F \big\vert_{\partial\Omega} \in \Hspace^{1}(\partial\Omega)$ and
  \begin{align*}
    \tangrad (F \big\vert_{\partial\Omega}) = \big(\nu \times (\grad F)\big\vert_{\partial\Omega}\big) \times \nu = \tantr \grad F.
  \end{align*}
\end{lemma}

\begin{proof}
  Let $k\colon \Gamma \to U$ be an arbitrary strong Lipschitz chart. Then $F \big\vert_{\partial\Omega} \circ k^{-1} = F \circ k^{-1}$ belongs to $\Hspace^{1}(U)$ by the chain rule. The tangential space at $\zeta \in \Gamma$ is given by the columns of $\diffd k^{-1}(k(\zeta))$. By construction the normal vector $\nu(\zeta)$ is orthogonal on this space.
  By \Cref{def:strong-H1-on-boundary} and the chain rule we have
  \begin{align*}
    \big(\tangrad F \big\vert_{\partial\Omega}\big)\big\vert_{\Gamma}
    = \big[\diffd (F \circ k^{-1}) (\diffd k^{-1})^{\dagger}\big]\trans \circ k
    = \big[(\diffd F \circ k^{-1})\diffd k^{-1} (\diffd k^{-1})^{\dagger}\big]\trans \circ k.
  \end{align*}
  Note that by \Cref{le:AAdagger-ortho-projection} the matrix $\diffd k^{-1} (\diffd k^{-1})^{\dagger} \circ k(\zeta)$ is the orthogonal projection on $\ran \diffd k^{-1}(k(\zeta))$. In particular this matrix is symmetric. Moreover, by \Cref{le:tantr-ortho-projetion-on-tangential-space} also $(\nu(\zeta) \times \cdot ) \times \nu(\zeta)$ is the orthogonal projection on the same space. Hence,
  \begin{align*}
    \big(\tangrad F \big\vert_{\partial\Omega}\big)\big\vert_{\Gamma}
    &= \big(\diffd k^{-1} (\diffd k^{-1})^{\dagger}\circ k\big) (\grad F \circ k^{-1})  \circ k
      = \big(\diffd k^{-1} (\diffd k^{-1})^{\dagger}\circ k\big) (\grad F)\big\vert_{\Gamma} \\
    &= \big(\nu \times (\grad F\big)\big\vert_{\Gamma}) \times \nu
      = (\tantr \grad F)\big\vert_{\Gamma}.
      \qedhere
  \end{align*}
\end{proof}

\begin{lemma}\label{le:weak-formulation-for-smooth-functions}
  Let $F \in \Cc(\R^{3})$ and $\Phi \in \Cc(\R^{3})^{3}$. Then
  \[
    \scprod*{\tantr \grad F}{\nu \times \Phi \big\vert_{\partial\Omega}}_{\Lp{2}(\partial \Omega)}
    = \scprod*{F \big\vert_{\partial\Omega}}{\nu \cdot (\rot \Phi)\big\vert_{\partial\Omega}}_{\Lp{2}(\partial\Omega)}.
  \]
\end{lemma}

\begin{proof}
  By the integration by parts formula for $\rot$ and $\div$-$\grad$ we have
  \begin{align*}
    \scprod*{\tantr \grad F}{\nu \times \Phi \big\vert_{\partial\Omega}}_{\Lp{2}(\partial \Omega)}
    &= \scprod{\grad F}{\rot \Phi}_{\Lp{2}(\Omega)} - \scprod{\underbrace{\rot \grad F}_{=\mathrlap{0}}}{\Phi}_{\Lp{2}(\Omega)}\\
    &= -\scprod{F}{\underbrace{\div \rot \Phi}_{=\mathrlap{0}}}_{\Lp{2}(\Omega)}
      + \scprod*{F \big\vert_{\partial\Omega}}{\nu \cdot (\rot \Phi)\big\vert_{\partial\Omega}}_{\Lp{2}(\partial\Omega)} \\
    &= \scprod*{F \big\vert_{\partial\Omega}}{\nu \cdot (\rot \Phi)\big\vert_{\partial\Omega}}_{\Lp{2}(\partial\Omega)}.
      \qedhere
  \end{align*}
\end{proof}

The previous lemma motivates the following alternative definition for $\Lp{2}(\partial\Omega)$ elements that possess a tangential gradient in a weak sense.

\begin{definition}\label{def:weak-H1-on-boundary}
  Let $\Omega$ be a strong Lipschitz domain. Then we say $f \in \tilde{\Hspace}^{1}(\partial\Omega)$,
  if there exists a $q \in \Lp{2}_{\tau}(\partial\Omega)$ such that for all $\Phi \in \Cc(\R^{3})^{3}$
  \begin{align*}
    \scprod*{q}{\nu \times \Phi \big\vert_{\partial\Omega}}_{\Lp{2}(\partial\Omega)}
    = \scprod*{f}{\nu \cdot (\rot \Phi)\big\vert_{\partial\Omega}}_{\Lp{2}(\partial\Omega)}.
  \end{align*}
  Moreover, we say $\wtangrad f = q$.
\end{definition}

Our goal will be to show that the space $\tilde{\Hspace}^{1}(\partial\Omega)$ coincides with $\Hspace^{1}(\partial\Omega)$. By \Cref{le:weak-formulation-for-smooth-functions} we see that $F \big\vert_{\partial\Omega} \in \tilde{\Hspace}^{1}(\partial\Omega)$ for every $F \in \Cc(\R^{3})$.

\begin{theorem}\label{th:smooth-functions-on-volume-dense-on-H1-boundary}
  The set $\dset[\big]{\Phi \big\vert_{\partial\Omega}}{\Phi \in \Cc(\R^{d})}$ is dense in $\Hspace^{1}(\partial\Omega)$ w.r.t.\ $\norm{\cdot}_{\Hspace^{1}(\partial\Omega)}$.
\end{theorem}

\begin{proof}
  By the definition of a strong Lipschitz domain we have for every $\zeta \in \partial\Omega$, a hyperplane $W$, a cylinder $C_{\epsilon,h}(\zeta)$ ($\epsilon$ and $h$ depend on $\zeta$), and a chart $k\colon \Gamma \to \ball_{\epsilon}(0)$, where $\Gamma = \partial\Omega \cap C_{\epsilon,h}(\zeta)$. Hence, we can cover $\partial\Omega$ by $\bigcup_{\zeta \in \partial\Omega} C_{\epsilon,h}(\zeta)$ and consequently there is a finite subcover $\bigcup_{i=1}^{m} C_{\epsilon_{i},h_{i}}(p_{i})$.
  We employ a partition of unity and obtain $(\alpha_{i})_{i=1}^{m}$, subordinate to this subcover, i.e.,
  \begin{align*}
    \alpha_{i} \in \Cc\big(C_{\epsilon_{i},h_{i}}(p_{i})\big), \quad \alpha_{i}(\zeta) \in [0,1], \quad\text{and}\quad \sum_{i=1}^{m} \alpha_{i}(\zeta) = 1
    \quad\text{for all}\quad \zeta \in
    \partial\Omega.
  \end{align*}
  For $f \in \Hspace^{1}(\partial\Omega)$ we define $f_{i} = \alpha_{i}\big\vert_{\partial\Omega} f$. It is straightforward to show that also $f_{i} \in \Hspace^{1}(\partial\Omega)$.
  We define $\Gamma_{i} = \partial\Omega \cap C_{\epsilon_{i},h_{i}}(p_{i})$ and the corresponding chart $k_{i}\colon \Gamma_{i} \to \ball_{\epsilon_{i}}(0) \subseteq \R^{d-1}$.
  Note that $\alpha_{i}\big\vert_{\partial\Omega}$ has compact support in $\Gamma_{i}$. Therefore, $f_{i} \circ k_{i}^{-1}$ has compact support in $\ball_{\epsilon_{i}}(0)$ and $f_{i} \circ k_{i}^{-1} \in \cH^{1}(\ball_{\epsilon_{i}}(0))$. This implies that there exists a sequence $(\varphi_{i,n})_{n\in\N}$ in $\Cc(\ball_{\epsilon_{i}}(0))$ that converges to $f_{i} \circ k_{i}^{-1}$ w.r.t.\ $\norm{\cdot}_{\Hspace^{1}(\ball_{\epsilon_{i}}(0))}$.

  We can define an extension of $\varphi_{i,n}$ on $\R^{d}$ with support on a strip by
  \begin{align}\label{eq:define-extension-on-strip}
    \Phi_{i,n}\left(
    \begin{bsmallmatrix}
      \zeta_{1} \\ \smatvdots \\ \zeta_{d}
    \end{bsmallmatrix}
    \right)
    = \varphi_{i,n}\left(
    \begin{bsmallmatrix}
      \zeta_{1} \\ \smatvdots \\ \zeta_{d-1}
    \end{bsmallmatrix}
    \right)
    \quad\text{or}\quad
    \Phi_{i,n}(\zeta) = \varphi_{i,n}(W\trans (\zeta - p_{i}))
  \end{align}
  in the general coordinates.
  Hence, $\Phi_{i,n} \in \conC^{\infty}(\R^{d})$.
  Note that
  we do not want that $\supp \Phi_{i,n}$ intersects $\partial\Omega$ outside of $\Gamma_{i}$. Thus, we multiply $\Phi_{i,n}$ by a suitable $\Cc$ cutoff function that is $1$ in a neighborhood of $\Gamma_{i}$ (for all $n\in\N$ the same cutoff function). Consequently, we even have $\Phi_{i,n} \in \Cc(\R^{d})$.

  By construction we have $\Phi_{i,n} \big\vert_{\Gamma_{i}} = \varphi_{i,n} \circ k_{i}$ and $\Phi_{i,n} \big\vert_{\partial\Omega} \to f_{i}$ in $\Hspace^{1}(\partial\Omega)$.
  Now we define $\Phi_{n} = \sum_{i=1}^{m} \Phi_{i,n} \in \Cc(\R^{d})$ and obtain $\Phi_{n}\big\vert_{\partial\Omega} \to f$ in $\Hspace^{1}(\partial\Omega)$.
\end{proof}

The density of $\dset{\Phi \big\vert_{\partial\Omega}}{\Phi \in \Cc(\R^{d})}$ implies that every $f \in \Hspace^{1}(\partial\Omega)$ is automatically also in $\tilde{\Hspace}^{1}(\partial\Omega)$, as the following corollary shows.

\begin{corollary}\label{th:weak-tangential-gradient-equals-strong-tangential-gradient}
  $\Hspace^{1}(\partial\Omega) \subseteq \tilde{\Hspace}^{1}(\partial\Omega)$ and $\tangrad f = \wtangrad f$ for all $f \in \Hspace^{1}(\partial\Omega)$.
\end{corollary}

\begin{proof}
  Let $f \in \Hspace^{1}(\partial\Omega)$. Then by \Cref{th:smooth-functions-on-volume-dense-on-H1-boundary} there exists a sequence $(F_{n})_{n\in\N}$ in $\Cc(\R^{3})$ such that $F_{n}\big\vert_{\partial\Omega} \to f$ w.r.t.\ $\norm{\cdot}_{\Hspace^{1}(\partial\Omega)}$. Hence, by \Cref{le:characterization-of-tangrad} and \Cref{le:weak-formulation-for-smooth-functions} we have for every $\Phi \in \Cc(\R^{3})$
  \begin{align*}
    \scprod{\tangrad f}{\nu \times \Phi}_{\Lp{2}(\partial\Omega)}
    &= \lim_{n\to\infty}\scprod{\tangrad F_{n}\big\vert_{\partial\Omega}}{\nu \times \Phi}_{\Lp{2}(\partial\Omega)}
    = \lim_{n \to \infty} \scprod{\tantr \grad F_{n}}{\nu \times \Phi}_{\Lp{2}(\partial\Omega)} \\
    &= \lim_{n \to \infty} \scprod{F_{n}\big\vert_{\partial\Omega}}{\nu \cdot (\rot \Phi)\big\vert_{\partial\Omega}}_{\Lp{2}(\partial\Omega)}
      = \scprod{f}{\nu \cdot (\rot \Phi)\big\vert_{\partial\Omega}}_{\Lp{2}(\partial\Omega)},
  \end{align*}
  which implies $f \in \tilde{\Hspace}^{1}(\partial\Omega)$ and $\tangrad f = \wtangrad f$.
\end{proof}

The next two lemmas are the foundation of the main result.
The second of these lemmas gives a lifting of a smooth function $\varphi$ on a flat domain in $\R^{2}$ to a smooth function $\Phi$ on $\R^{3}$ such that the twisted tangential trace of the lifting $\Phi$ equals the tangential field that corresponds to $\varphi$ (i.e., $\diffd k^{-1} \varphi$). This automatically gives an identity for the $\R^{2}$ divergence of $\varphi$ and in terms of $\Phi$.

\begin{lemma}\label{le:tanxtr-as-linear-combination-of-tangential-vectors}
  Let $k\colon \Gamma \to U$ be a strong Lipschitz chart. Then for every $\varphi \in \Cc(U)$ we have
  \begin{align*}
    \frac{1}{\sqrt{\det\big((\diffd k^{-1})\trans \diffd k^{-1}\big)}} \diffd k^{-1} \varphi
    = (\nu \circ k^{-1})\times
    \begin{bmatrix}
      \varphi_{2} \\
      -\varphi_{1} \\
      0
    \end{bmatrix}.
  \end{align*}
\end{lemma}

\begin{proof}
  The following calculation proves the claim
  \begin{align*}
    \MoveEqLeft(\nu \circ k^{-1})\times
    \begin{bmatrix}
      \varphi_{2} \\ -\varphi_{1} \\ 0
    \end{bmatrix} \\[1.2ex]
    &=
    \begin{bmatrix}
      0 & -\nu_{3} & \nu_{2} \\
      \nu_{3} & 0 & -\nu_{1} \\
      -\nu_{2} & \nu_{1} & 0
    \end{bmatrix} \circ k^{-1}
                           \begin{bmatrix}
                             \varphi_{2} \\ -\varphi_{1} \\ 0
                           \end{bmatrix}
    =
    \varphi_{2}
    \begin{bmatrix}
      0 \\
      \nu_{3} \\
      -\nu_{2} \\
    \end{bmatrix} \circ k^{-1}
    -
    \varphi_{1}
    \begin{bmatrix}
      -\nu_{3} \\
      0 \\
      \nu_{1}
    \end{bmatrix} \circ k^{-1}
    \\[1.2ex]
    &=
      \frac{1}{\sqrt{1 + \norm{\grad a}^{2}}}\left(
      \varphi_{1}
      \begin{bmatrix}
        1 \\ 0 \\ \partial_{1} a
      \end{bmatrix}
    +
    \varphi_{2}
    \begin{bmatrix}
      0 \\ 1 \\ \partial_{2} a
    \end{bmatrix}
    \right)
    = \frac{1}{\sqrt{\det\big((\diffd k^{-1})\trans \diffd k^{-1}\big)}} \diffd k^{-1} \varphi.
      \qedhere
  \end{align*}
\end{proof}

\begin{lemma}\label{le:lift-through-chart}
  Let $\Gamma \subseteq \partial\Omega$ be a chart domain and $k\colon \Gamma \to U$ a strong Lipschitz chart.
  Then for every $\varphi \in \Cc(U)$ there exists a $\Phi \in \Cc(\R^{3})$ such that we have
  \begin{equation*}
    \Phi \big\vert_{\Gamma} =
    \begin{bmatrix}
      \varphi_{2} \\
      -\varphi_{1} \\
      0
    \end{bmatrix}
    \circ k
    \quad\text{and}\quad
    \Phi \big\vert_{\partial\Omega \setminus \Gamma} = 0
  \end{equation*}
  on the boundary, and
  \begin{align}\label{eq:lifting-identity-tanxtr}
    \diffd k^{-1} \varphi &= \sqrt{\det \big( (\diffd k^{-1})\trans \diffd k^{-1}\big)} (\nu \times \Phi) \circ k^{-1}, \\
    \label{eq:lifting-identity-divergence}
    \div_{\R^{2}} \varphi &= -\sqrt{\det \big( (\diffd k^{-1})\trans \diffd k^{-1}\big)} (\nu \cdot \rot \Phi) \circ k^{-1}.
  \end{align}
\end{lemma}

\begin{proof}
  We define
  \begin{align*}
    \hat{\Phi} \colon
    \mapping{U \times \R \subseteq \R^{3}}{\C^{3}}{
    \begin{bmatrix}
      \zeta_{1} \\ \zeta_{2} \\ \zeta_{3}
    \end{bmatrix}
    }{
    \begin{bmatrix}
      \varphi_{2}(\zeta_{1},\zeta_{2}) \\
      -\varphi_{1}(\zeta_{1},\zeta_{2}) \\
      0
    \end{bmatrix}.
    }
  \end{align*}
  Since $\varphi$ has compact support in $U$ we can extend $\hat{\Phi}$ outside of $U \times \R$ by $0$. Moreover we choose an $\epsilon > 0$ such that the ball with radius $2\epsilon$ around $\Gamma$ satisfies
  \[
    \ball_{2\epsilon}(\Gamma) \cap \supp \hat{\Phi} \cap (\partial\Omega \setminus \Gamma) = \emptyset.
  \]
  Finally, we choose a cut-off function $\chi \in \Cc(\R^{3})$ such that $\chi \big\vert_{\ball_{\epsilon}(\Gamma)} = 1$ and $\chi \big\vert_{\ball_{2\epsilon}(\Gamma)^{\complement}} = 0$ and we define $\Phi \coloneqq \chi \hat{\Phi}$. Hence, $\Phi \big\vert_{\partial\Omega \setminus \Gamma} = 0$.


  By construction we have
  \(
  \Phi \circ k^{-1}(x_{1},x_{2}) =
  \begin{bsmallmatrix}
    \varphi_{2}(x_{1},x_{2}) \\
    -\varphi_{1}(x_{1},x_{2}) \\
    0
  \end{bsmallmatrix}
  \).
  Thus, \Cref{le:tanxtr-as-linear-combination-of-tangential-vectors} implies~\eqref{eq:lifting-identity-tanxtr}.

  Note that locally around $\Gamma$ we have
  \(
    \Phi =
    \begin{bsmallmatrix}
      \varphi_{2} \\
      -\varphi_{1} \\
      0
    \end{bsmallmatrix}
  \)
  and $\partial_{3} \Phi$ = 0.
  Hence, we have
  \begin{align*}
    -\sqrt{1 + \norm{\grad a}^{2}} \, \nu \cdot \rot \Phi
    &=
      \begin{bmatrix}
        \partial_{1} a \\
        \partial_{2} a \\
        -1
      \end{bmatrix}
      \cdot
      \begin{bmatrix}
        0 & -\partial_{3} & \partial_{2} \\
        \partial_{3} & 0 & -\partial_{1} \\
        -\partial_{2} & \partial_{1} & 0
      \end{bmatrix}
      \begin{bmatrix}
        \varphi_{2} \\
        - \varphi_{1} \\
        0
      \end{bmatrix}
    \\
    &=
      \begin{bmatrix}
        \partial_{1} a \\
        \partial_{2} a \\
        -1
      \end{bmatrix}
      \cdot
    \begin{bmatrix}
      \partial_{3} \varphi_{1} \\
      \partial_{3} \varphi_{2} \\
      -\partial_{1} \varphi_{1} - \partial_{2} \varphi_{2}
    \end{bmatrix}
    \\
    &=
      \partial_{1} a \underbrace{\partial_{3} \varphi_{1}}_{=\mathrlap{0}}
      \mathclose{} + \partial_{2} a \underbrace{\partial_{3} \varphi_{2}}_{=\mathrlap{0}}
      \mathclose{} + \partial_{1} \varphi_{1} + \partial_{2} \varphi_{2}
      = \div_{\R^{2}} \varphi.
      \qedhere
  \end{align*}
\end{proof}


Finally, we come to the main result, that proves that both presented approaches (\Cref{def:strong-H1-on-boundary} and \Cref{def:weak-H1-on-boundary}) to the first order Sobelev space on $\partial\Omega$ lead to the same space.

\begin{theorem}\label{th:weak-equals-strong-H1-on-boundary}
  $\tilde{\Hspace}^{1}(\partial\Omega) = \Hspace^{1}(\partial\Omega)$ and $\wtangrad f = \tangrad f$ for all $f \in \Hspace^{1}(\partial\Omega)$.
\end{theorem}

\begin{proof}
  Let $f \in \tilde{\Hspace}^{1}(\partial\Omega)$, i.e., there exists a $q \in \Lp{2}_{\tau}(\partial\Omega)^{3}$ such that
  \begin{equation}\label{eq:p-in-weak-H1-on-the-boundary}
    \scprod{q}{\nu \times \Phi \big\vert_{\partial\Omega}}_{\Lp{2}(\partial\Omega)} = \scprod{f}{\nu \cdot (\rot \Phi) \big\vert_{\partial\Omega}}_{\Lp{2}(\partial\Omega)}
    \quad\text{for all}\quad \Phi \in \Cc(\R^{3})^{3}
  \end{equation}
  Let $\Gamma \subseteq \partial\Omega$ be a chart domain, $U\subseteq \R^{2}$ open and $k \colon \Gamma \to U$ a strong Lipschitz chart.
  For an arbitrary $\varphi \in \Cc(U)$ we define $\Phi$ as in \Cref{le:lift-through-chart}. Then we have
  \begin{align*}
    \MoveEqLeft[3]
    -\scprod{f\circ k^{-1}}{\div_{\R^{2}} \varphi}_{\Lp{2}(U)} \\
    &\stackrel{\mathclap{\eqref{eq:lifting-identity-divergence}}}{=}
      \scprod*{f \circ k^{-1}}{\sqrt{\det\big((\diffd k^{-1})\trans \diffd k^{-1}\big)} (\nu \cdot \rot \Phi) \circ k^{-1}}_{\Lp{2}(U)} \\
    &= \scprod*{f}{\nu \cdot \rot \Phi\big\vert_{\partial\Omega}}_{\Lp{2}(\Gamma)}
    = \scprod*{f}{\nu \cdot \rot \Phi\big\vert_{\partial\Omega}}_{\Lp{2}(\partial\Omega)}
      \stackrel{\eqref{eq:p-in-weak-H1-on-the-boundary}}{=} \scprod*{q}{\nu \times \Phi\big\vert_{\partial\Omega}}_{\Lp{2}(\partial\Omega)} \\
    &= \scprod*{q}{\nu \times \Phi\big\vert_{\partial\Omega}}_{\Lp{2}(\Gamma)}
    = \scprod*{q\circ k^{-1}}{\sqrt{\det \big((\diffd k^{-1})\trans \diffd k^{-1}\big)} (\nu \times \Phi)\circ k^{-1} }_{\Lp{2}(U)} \\
    &\stackrel{\mathclap{\eqref{eq:lifting-identity-tanxtr}}}{=} \scprod*{q \circ k^{-1}}{\diffd k^{-1} \varphi}_{\Lp{2}(U)}
    = \scprod*{(\diffd k^{-1})\trans (q\circ k^{-1})}{\varphi}_{\Lp{2}(U)}
  \end{align*}
  Hence, $f \circ k^{-1} \in \Hspace^{1}(U)$. Since this is true for any chart $k$ we conclude $f \in \Hspace^{1}(\partial\Omega)$.

  For $f \in \Hspace^{1}(\partial\Omega)$ we conclude $\wtangrad f = \tangrad f$ from \Cref{th:weak-tangential-gradient-equals-strong-tangential-gradient}.
\end{proof}

\section{Back to the original question}\label{sec:back-to-original-question}

For this last section we assume that the reader has some basic knowledge about $\Hspace(\rot,\Omega)$ and $\Hspace(\div,\Omega)$, see, e.g., \cite[Section~3.5]{monk}.

\begin{definition}
  We say $G \in \Hspace(\rot,\Omega)$ possesses a (weak) $\Lp{2}$ tangential trace, if there exists a $q \in \Lp{2}_{\tau}(\partial\Omega)$ such that
  \begin{align*}
    \scprod{G}{\rot \Phi}_{\Lp{2}(\Omega)} - \scprod{\rot G}{\Phi}_{\Lp{2}(\Omega)}
    = \scprod{q}{\nu \times \Phi \big\vert_{\partial\Omega}}_{\Lp{2}_{\tau}(\partial\Omega)}
    \quad\text{for all}\quad \Phi \in \Cc(\R^{3})^{3}.
  \end{align*}
  We say then $q$ is the tangential trace of $G$, i.e., $\tantr G = q$.
\end{definition}

\begin{theorem}\label{th:weak-tangential-trace-of-gradient-implies-H1-on-boundary}
  Let $F \in \Hspace^{1}(\Omega)$ be such that $\grad F$ possesses a (weak) $\Lp{2}$ tangential trace. Then $F \big\vert_{\partial\Omega} \in \Hspace^{1}(\partial\Omega)$ and $\tantr \grad F = \tangrad F \big\vert_{\partial\Omega}$.
\end{theorem}

\begin{proof}
  Let $q \in \Lp{2}_{\tau}(\partial\Omega)$ be such that $q = \tantr \grad F$. By the integration by parts formula for $\rot$ and $\div$-$\grad$, we have for an arbitrary $\Phi \in \Cc(\R^{3})^{3}$
  \begin{align*}
    \scprod{q}{\nu \times \Phi \big\vert_{\partial\Omega}}_{\Lp{2}_{\tau}(\partial\Omega)}
    &= \scprod{\grad F}{\rot \Phi}_{\Lp{2}(\Omega)} - \scprod{\underbrace{\rot \grad F}_{=\mathrlap{0}}}{\Phi}_{\Lp{2}(\Omega)} \\
    &= - \scprod{F}{\underbrace{\div \rot \Phi}_{=\mathrlap{0}}}_{\Lp{2}(\Omega)}
      + \scprod*{F \big\vert_{\partial\Omega}}{\nu \cdot (\rot \Phi)\big\vert_{\partial\Omega}}_{\Lp{2}(\partial\Omega)} \\
    &= \scprod*{F \big\vert_{\partial\Omega}}{\nu \cdot (\rot \Phi)\big\vert_{\partial\Omega}}_{\Lp{2}(\partial\Omega)}.
  \end{align*}
  Hence, $F \big\vert_{\partial\Omega}$ satisfies all requirements of \Cref{def:weak-H1-on-boundary}, which implies, by \Cref{th:weak-equals-strong-H1-on-boundary}, $F \big\vert_{\partial\Omega} \in \Hspace^{1}(\partial\Omega)$. In particular we have
  \begin{equation*}
    \tantr \grad F \big\vert_{\partial\Omega} = q = \wtangrad F \big\vert_{\partial\Omega} = \tangrad F \big\vert_{\partial\Omega}.
    \qedhere
  \end{equation*}
\end{proof}

\section{Conclusion}

With \Cref{th:weak-equals-strong-H1-on-boundary} we have shown that both presented approaches to $\Hspace^{1}(\partial\Omega)$ agree. Moreover, \Cref{th:weak-tangential-trace-of-gradient-implies-H1-on-boundary} answers the question about the validity of \Cref{cl:motivation}%
, that started the whole discussion,
positively.
Hence, we provide the details that are used in \cite[Proof of Thm.~2]{Cos90}, \cite[Section Le cas tridimensionnel]{BeBeCoDa97}, \cite[Proof of Lem.~3.53]{monk}, and \cite[Proof of Thm.~5.1]{BuCoSh02}.

\section*{Acknowledgement}

We thank Dirk Pauly and Martin Costabel for the discussions about \Cref{cl:motivation}.

\appendix
\section{Details for general hyperplanes}\label{sec:general-charts}

Note that in the setting with a general hyperplane $W = \spn\set{w_{1},w_{2}}$, where $w_{1}$ and $w_{2}$ are an orthonormal basis of $W$, and its normal vector $v$ we have
\begin{equation*}
  k^{-1} \colon
  \mapping{U \subseteq \R^{2}}{\Gamma}{(x_{1},x_{2})}{p + x_{1} w_{1} + x_{2} w_{2} + a(x_{1},x_{2})v.}
\end{equation*}
Hence,
\begin{equation*}
  \diffd k^{-1} =
  \begin{bmatrix}
    w_{1} + \partial_{1} a v & w_{2} + \partial_{2} a v
  \end{bmatrix}
\end{equation*}
and the normal vector on the tangential space is given by
\begin{equation*}
  \nu(k^{-1}(x_{1},x_{2})) = \frac{1}{\sqrt{1 + \norm{\grad a}^{2}}} (-\partial_{1} a w_{1} - \partial_{2} a w_{2} + v).
\end{equation*}
Moreover, we have
\begin{equation*}
  (\diffd k^{-1})\trans \diffd k^{-1} =
  \begin{bmatrix}
    1 + (\partial_{1}a)^{2} & \partial_{1} a\partial_{2}a \\
    \partial_{1} a\partial_{2}a & 1 + (\partial_{2}a)^{2}
  \end{bmatrix}
  =
  \begin{bmatrix}
    1 & 0 \\
    0 & 1
  \end{bmatrix}
  +
  \begin{bmatrix}
    \partial_{1} a \\ \partial_{2} a
  \end{bmatrix}
  \begin{bmatrix}
    \partial_{1} a & \partial_{2} a
  \end{bmatrix}.
\end{equation*}
and therefore \Cref{le:functional-determinant-of-chart} follows also for general strong Lipschitz charts:
\begin{equation*}
  \det\big((\diffd k^{-1})\trans \diffd k^{-1}\big) = 1 + \norm{\grad a}^{2}.
\end{equation*}

We show the modified \Cref{le:tanxtr-as-linear-combination-of-tangential-vectors,le:lift-through-chart} for general strong Lipschitz charts

\begin{lemma}\label{le:tanxtr-as-linear-combination-of-tangential-vectors-general}
  For $\varphi \in \Cc(U)^{2}$ we have
  \begin{equation*}
    \diffd k^{-1} \varphi = \sqrt{\det\big((\diffd k^{-1})\trans \diffd k^{-1}\big)} \,(\nu \circ k^{-1}) \times (\varphi_{2} w_{1} - \varphi_{1}w_{2}),
  \end{equation*}
  where the orthogonal basis $\set{w_{1},w_{2},v}$ is chosen such that $w_{1} \times w_{2} = v$ (if this is not already true we relabel $w_{1}$ and $w_{2}$).
\end{lemma}

Note that $w_{1} \times w_{2} = v$ implies
\begin{align*}
  w_{1} \times v = -w_{2} \quad\text{and}\quad w_{2} \times v = w_{1}.
\end{align*}

\begin{proof}
  Note that
  \[
    \sqrt{\det \big((\diffd k^{-1})\trans \diffd k^{-1}\big)} (\nu \circ k^{-1}) = -\partial_{1}a w_{1} - \partial_{2} a w_{2} + v.
  \]
  Therefore, the following proves the claim:
  \begin{multline*}
    (-\partial_{1}a w_{1} - \partial_{2} a w_{2} + v) \times (\phi_{2} w_{1} - \phi_{1} w_{2}) \\
    = (\partial_{2}a v + w_{2})\phi_{2} + (\partial_{1}a v + w_{1})\phi_{1}
    = \diffd k^{-1} \phi.
    \tag*{\qedhere}
  \end{multline*}
\end{proof}

\begin{lemma}
  Let $\Gamma \subseteq \partial\Omega$ be a chart domain and $k\colon \Gamma \to U$ a strong Lipschitz chart.
  Then for every $\varphi \in \Cc(U)$ there exists a $\Phi \in \Cc(\R^{3})$ such that we have
  \begin{align*}
    \Phi \big\vert_{\Gamma}
    = W
    \begin{bmatrix}
      \varphi_{2} \\
      -\varphi_{1}
    \end{bmatrix}
    \circ k
    =
    (\varphi_{2}\circ k) w_{1} - (\varphi_{1} \circ k) w_{2}
    \quad\text{and}\quad \Phi \big\vert_{\partial\Omega \setminus \Gamma} = 0
  \end{align*}
  on the boundary, and
  \begin{align}\label{eq:lifting-identity-tanxtr-general}
    \diffd k^{-1} \varphi &= \sqrt{\det \big( (\diffd k^{-1})\trans \diffd k^{-1}\big)} (\nu \times \Phi) \circ k^{-1}, \\
    \label{eq:lifting-identity-divergence-general}
    \div_{\R^{2}} \varphi &= -\sqrt{\det \big( (\diffd k^{-1})\trans \diffd k^{-1}\big)} (\nu \cdot \rot \Phi) \circ k^{-1}.
  \end{align}
\end{lemma}

\begin{proof}
  We define
  $\hat{\Phi} \in \conC^{\infty}(\R^{3})^{3}$ by
  \begin{align*}
    \hat{\Phi}(\zeta)
    = W
    \begin{bmatrix}
      \varphi_{2} \\ -\varphi_{1}
    \end{bmatrix}
    \big(W\trans(\zeta - p)\big)
    = \varphi_{2}(W\trans (\zeta - p)) w_{1} - \varphi_{1}(W\trans (\zeta - p)) w_{2},
  \end{align*}
  where $W \in \R^{3\times 2}$ is the matrix containing the vectors $w_{1}$ and $w_{2}$ as rows, i.e.,
  \(
  W =
  \begin{bmatrix}
    w_{1} & w_{2}
  \end{bmatrix}
  \).
  Finally, we define $\Phi \in \Cc(\R^{3})^{3}$ by $\chi \hat{\Phi}$ where $\chi \in \Cc(\R^{3})$ is such that in a small neighborhood of $\Gamma$ $\chi = 1$ and $\Phi \big\vert_{\partial\Omega\setminus \Gamma} = 0$.
  Basically, by construction we have
  \(
  \Phi \big\vert_{\Gamma} = W
  \begin{bsmallmatrix}
    \varphi_{2} \\ - \varphi_{1}
  \end{bsmallmatrix}
  \circ k
  \).
  Hence, we have $\Phi \circ k^{-1} = \varphi_{2} w_{1} - \varphi_{1} w_{2}$ and \Cref{le:tanxtr-as-linear-combination-of-tangential-vectors-general} gives \eqref{eq:lifting-identity-tanxtr-general}

  For an arbitrary $f \in \Cc(U)$ we have
  \begin{align*}
    -\scprod{f}{\div_{\R^{2}} \varphi}_{\Lp{2}(U)}
    &= \scprod{\grad_{\R^{2}} f}{\varphi}_{\Lp{2}(U)}
      = \scprod*{\big((\diffd k^{-1})^{\dagger}\big)\trans \grad_{\R^{2}} f}{\diffd k^{-1} \varphi}_{\Lp{2}(U)} \\
    &= \scprod[\Big]{\big[{(\diffd k^{-1})^{\dagger}}\trans \grad_{\R^{2}} f\big] \circ k}{\Big[\tfrac{1}{\sqrt{\det((\diffd k^{-1})\trans \diffd k^{-1})}}\diffd k^{-1} \varphi\Big] \circ k}_{\Lp{2}(\Gamma)} \\
    &=
    \scprod*{\tangrad(f \circ k)}{\nu \times \Phi \big\vert_{\Gamma}}_{\Lp{2}(\Gamma)}
    = \scprod*{f \circ k}{\nu \cdot (\rot \Phi)\big\vert_{\Gamma}}_{\Lp{2}(\Gamma)} \\
    &= \scprod*{f}{(\nu \cdot \rot \Phi ) \circ k^{-1}\sqrt{\det\big((\diffd k^{-1})\trans \diffd k^{-1}\big)}}_{\Lp{2}(U)}
  \end{align*}
  By density of $\Cc(U)$ in $\Lp{2}(U)$ we obtain \eqref{eq:lifting-identity-divergence-general}.
\end{proof}


\section{Independence of the charts}

Note that for two strong Lipschitz charts $k_{1}\colon \Gamma_{1} \to U_{1}$, $k_{2}\colon \Gamma_{2} \to U_{2}$ with overlapping chart domains (i.e., $\Gamma_{1} \cap \Gamma_{2} \neq \emptyset$) we have that the columns of $\diffd k_{1}^{-1}(k_{1}(\zeta))$ and the columns of $\diffd k_{2}^{-1}(k_{2}(\zeta))$ span the same linear subspace of $\R^{d}$ for a.e.\ $\zeta \in \Gamma_{1} \cap \Gamma_{2}$, namely the \emph{tangential space} of $\partial\Omega$ at $\zeta$. The next lemma will specify this.

\begin{lemma}
  Let $k_{1} \colon \Gamma_{1} \to U_{1}$ and $k_{2}\colon \Gamma_{2} \to U_{2}$ be strong Lipschitz charts.
  Then
  \begin{align*}
    \ran \big[\diffd k_{1}^{-1}\big(k_{1}(\zeta)\big)\big]
    = \ran \big[\diffd k_{2}^{-1}\big(k_{2}(\zeta)\big)\big]
    \quad\text{for a.e.}\quad \zeta \in \Gamma_{1} \cap \Gamma_{2}.
  \end{align*}
  Moreover,
  \begin{align}\label{eq:change-of-charts-2}
    (\diffd k_{1}^{-1})^{\dagger} \circ (k_{1} \circ k_{2}^{-1})\, \diffd k_{2}^{-1} = \diffd(k_{1} \circ k_{2}^{-1}).
  \end{align}
\end{lemma}

\begin{proof}
  The first assertion follows from
  \begin{align}\label{eq:change-of-charts}
    \diffd k_{2}^{-1} = \diffd (k_{1}^{-1} \circ k_{1} \circ k_{2}^{-1}) = (\diffd k_{1}^{-1})\circ (k_{1} \circ k_{2}^{-1})\, \diffd(k_{1} \circ k_{2}^{-1})
  \end{align}
  and the fact that $\diffd(k_{1} \circ k_{2}^{-1})(\zeta)$ is a regular matrix for a.e.\ $\zeta \in \Gamma_{1} \cap \Gamma_{2}$.
  Multiplying both side of~\eqref{eq:change-of-charts} with $(\diffd k_{1}^{-1})^{\dagger} \circ (k_{1} \circ k_{2}^{-1})$ implies~\eqref{eq:change-of-charts-2}.
\end{proof}



\begin{lemma}\label{le:dk-dk-dagger-concides-for-different-charts}
  Let $k_{1}\colon \Gamma_{1} \to U_{1}$, $k_{2}\colon \Gamma_{2} \to U_{2}$ strong Lipschitz charts. Then for a.e.\ $\zeta \in \Gamma_{1} \cap \Gamma_{2}$ the following holds
  \begin{align*}
    (\diffd k_{1}^{-1})(\diffd k_{1}^{-1})^{\dagger} \circ k_{1}(\zeta)
    = (\diffd k_{2}^{-1})(\diffd k_{2}^{-1})^{\dagger} \circ k_{2}(\zeta).
  \end{align*}
\end{lemma}

\begin{proof}
  Note that for a.e.\ $\zeta \in \Gamma_{1} \cap \Gamma_{2}$ we have $\ran [\diffd k_{1}^{-1}(k_{1}(\zeta))] = \ran [\diffd k_{2}^{-1}(k_{2}(\zeta))]$. By \Cref{le:AAdagger-ortho-projection} $(\diffd k_{1}^{-1})(\diffd k_{1}^{-1})^{\dagger} \circ k_{1}(\zeta)$ is the orthogonal projection on $\ran [\diffd k_{1}^{-1}(k_{1}(\zeta))]$ and $(\diffd k_{2}^{-1})(\diffd k_{2}^{-1})^{\dagger} \circ k_{2}(\zeta)$ is the orthogonal projection on $\ran [\diffd k_{2}^{-1}(k_{2}(\zeta))]$. Since these ranges coincide we conclude the assertion.
\end{proof}

Sometimes it is more convenient to work with the boundary derivative $\diffd_{\tau}$ instead of the the tangential gradient $\tangrad$. This derivative is given by $\diffd_{\tau} f = (\tangrad f)\trans$ or locally by
\(
  (\diffd_{\tau} f)\big\vert_{\Gamma} = \big[\diffd(f \circ k^{-1}) (\diffd k^{-1})^{\dagger} \big] \circ k.
\)

\begin{proposition}\label{th:tangential-derivative-independent-of-charts}
  Let $f \in \Hspace^{1}(\Omega)$. Then $\tangrad f$ and $\diffd_{\tau} f$ are independent of the charts.
\end{proposition}

\begin{proof}
  Let $k_{1}$ and $k_{2}$ be two charts with overlapping chart domains. Then we have
  \begin{align*}
    (\diffd_{\tau} f)\big\vert_{\Gamma_{1} \cap \Gamma_{2}}
    &= \big[ \diffd (f \circ k_{2}^{-1}) (\diffd k_{2}^{-1})^{\dagger} \big]\circ k_{2}
      = \big[ \diffd (f \circ k_{1}^{-1} \circ k_{1} \circ k_{2}^{-1}) (\diffd k_{2}^{-1})^{\dagger} \big]\circ k_{2}
      \\
    &= \big[ \diffd (f \circ k_{1}^{-1}) \circ (k_{1} \circ k_{2}^{-1})\, \underbrace{\diffd(k_{1} \circ k_{2}^{-1})}_{\stackrel{\eqref{eq:change-of-charts-2}}{=}\mathrlap{(\diffd k_{1}^{-1})^{\dagger} \circ (k_{1} \circ k_{2}^{-1}) \diffd k_{2}^{-1}}} (\diffd k_{2}^{-1})^{\dagger} \big]\circ k_{2} \\
    &= \big[ \diffd (f \circ k_{1}^{-1}) \circ (k_{1} \circ k_{2}^{-1}) (\diffd k_{1}^{-1})^{\dagger} \circ (k_{1} \circ k_{2}^{-1})
      \underbrace{\diffd k_{2}^{-1} (\diffd k_{2}^{-1})^{\dagger}}_{\mathllap{[\diffd k_{1}^{-1} (\diffd k_{1}^{-1})^{\dagger}] \circ (k_{1} \circ k_{2}^{-1})}\stackrel{\text{L.\ref{le:dk-dk-dagger-concides-for-different-charts}}}{=}} \big]\circ k_{2}
    \\
    \intertext{Note that $A^{\dagger} A A^{\dagger} = A^{\dagger}$.}
    &= \big[ \diffd(f \circ k_{1}) \circ (k_{1} \circ k_{2}^{-1}) (\diffd k_{1}^{-1})^{\dagger} \circ (k_{1} \circ k_{2}^{-1})\big] \circ k_{2} \\
    &=\big[ \diffd(f \circ k_{1})  (\diffd k_{1}^{-1})^{\dagger} \big] \circ k_{1}.
      \qedhere
  \end{align*}
\end{proof}

\begin{proposition}\label{th:surface-measure-independent-of-charts}
  The surface measure on $\partial\Omega$ is independent of the partition and the charts.
\end{proposition}

\begin{proof}
  It is enough to show that two charts $k_{1}\colon \Gamma_{1} \to U_{1}$ and $k_{2}\colon \Gamma_{2} \to U_{2}$ with intersecting chart domains define the same surface measure on the intersection $\Gamma_{1} \cap \Gamma_{2}$. The rest can be done by intersecting the two partitions.

  We define the mapping
  \begin{align*}
    T \colon \mapping{k_{2}(\Gamma_{1} \cap \Gamma_{2}) \subseteq U_{2}}{k_{1}(\Gamma_{1} \cap \Gamma_{2}) \subseteq U_{1}}{x}{(k_{1} \circ k_{2}^{-1})(x),}
  \end{align*}
  which gives a bijective bi-Lipschitz continuous mapping. Note that by the chain rule we have
  \begin{align*}
    \diffd k_{2}^{-1} = \diffd (k_{1}^{-1} \circ k_{1} \circ k_{2}^{-1})
    = (\diffd k_{1}^{-1}) \circ (k_{1} \circ k_{2}^{-1}) \diffd (k_{1} \circ k_{2}^{-1})
    = (\diffd k_{1}^{-1}) \circ T \diffd T.
  \end{align*}
  Moreover, by properties of the determinant we have
  \begin{align*}
    \abs{\det \diffd T}  \sqrt{\det(\diffd k_{1}^{-1} \circ T)\trans (\diffd k_{1}^{-1} \circ T)}
    &= \sqrt{\det (\diffd T)\trans \diffd T}  \sqrt{\det(\diffd k_{1}^{-1} \circ T)\trans (\diffd k_{1}^{-1} \circ T)} \\
    &= \sqrt{\det (\diffd T)\trans (\diffd k_{1}^{-1} \circ T)\trans (\diffd k_{1}^{-1}\circ T) \diffd T} \\
    &= \sqrt{\det ((\diffd k_{1}^{-1} \circ T)\diffd T)\trans ((\diffd k_{1}^{-1}\circ T) \diffd T)} \\
    &= \sqrt{\det (\diffd k_{2}^{-1})\trans \diffd k_{2}^{-1}}.
  \end{align*}
  Now for $\Upsilon \subseteq \Gamma_{1} \cap \Gamma_{2}$ we have by change of variables
  \begin{align*}
    \int_{k_{1}(\Upsilon)} \sqrt{\det (\diffd k_{1}^{-1})\trans \diffd k_{1}^{-1}} \dx[\uplambda_{d-1}]
    &= \int_{T^{-1}(k_{1}(\Upsilon))} \sqrt{\det (\diffd k_{1}^{-1})\trans \diffd k_{1}^{-1}}\circ T \abs{\det \diffd T} \dx[\uplambda_{d-1}] \\
    &= \int_{k_{2}(\Upsilon)} \sqrt{\det (\diffd k_{2}^{-1})\trans \diffd k_{2}^{-1}} \dx[\uplambda_{d-1}].
  \end{align*}
  Hence, the surface measure $\mu(\Upsilon)$ is independent of the charts.
\end{proof}

\section{Some auxiliary lemmas}

\begin{lemma}\label{le:determinant-of-I-plus-vvT}
  Let $v \in \R^{d}$ then
  \begin{equation*}
    \det(I + vv\trans) = 1 + \norm{v}^{2}.
  \end{equation*}
\end{lemma}

\begin{proof}
  Note that the determinant of a matrix equals the product of all eigenvalues. Let $b_{1},\dots, b_{d-1}$ denote an orthonormal basis of $\set{v}^{\perp}$. Then we can easily see that each $b_{i}$ is an eigenvector of $I + vv\trans$ with eigenvalue $1$. Furthermore, $(I + vv\trans) v = (1 + \norm{v}^{2})v$ implies that $v$ is an eigenvector with eigenvalue $1+\norm{v}^{2}$. Hence, we have found all eigenvalues and consequently the determinant equals $1 + \norm{v}^{2}$.
\end{proof}

\begin{lemma}\label{le:tantr-ortho-projetion-on-tangential-space}
  For $w \in \C^{3}$ with $\norm{w} = 1$ the mapping $A\colon v \mapsto (w \times v) \times w$ is the orthogonal projection on the orthogonal complement of $\spn\set{w}$.
\end{lemma}

\begin{proof}
  Note that $(w\times v) \times w = -w \times (w \times v)$ and
  $w \times v =
  \begin{bsmallmatrix}
    0 & -w_{3} & w_{2} \\
    w_{3} & 0 & -w_{1} \\
    -w_{2} & w_{1} & 0
  \end{bsmallmatrix}
  v$. Therefore,
  \begin{align*}
    (w \times v) \times w &=
    -
    \begin{bmatrix}
      0 & -w_{3} & w_{2} \\
      w_{3} & 0 & -w_{1} \\
      -w_{2} & w_{1} & 0
    \end{bmatrix}^{2}
    v
    =
    \begin{bmatrix}
      w_{2}^{2} + w_{3}^{2} & -w_{1} w_{2} & -w_{1} w_{3} \\
      -w_{1} w_{2} & w_{1}^{2} + w_{3}^{2} & -w_{2} w_{3} \\
      -w_{1} w_{3} & -w_{2} w_{3} & w_{1}^{2} + w_{2}^{2}
    \end{bmatrix}
    v \\
    \intertext{Since $\norm{w} = 1$ we further have}
                          &=
                            \left(
                            \begin{bmatrix}
                              1 & 0 & 0 \\
                              0 & 1 & 0 \\
                              0 & 0 & 1
                            \end{bmatrix}
                            -
                            \begin{bmatrix}
                              w_{1}^{2} & w_{1}w_{2} & w_{1} w_{3} \\
                              w_{1} w_{2} & w_{2}^{2} & w_{2} w_{3} \\
                              w_{1} w_{3} & w_{2} w_{3} & w_{3}^{2}
                            \end{bmatrix}
                            \right)
                            v
                            =
                            (I - ww\trans)v,
  \end{align*}
  which shows the claim.
\end{proof}

\begin{lemma}\label{le:AAdagger-ortho-projection}
  Let $A$ be an injective matrix and $A^{\dagger} = (A\trans A)^{-1} A\trans$ its Moore-Penrose inverse. Then $A A^{\dagger}$ is the orthogonal projection on $\ran A$.
\end{lemma}

\begin{proof}
  Note that $\ker A\trans = (\ran A)^{\perp}$, $\ker A = (\ran A\trans)^{\perp}$, and $\ker A^{\dagger} = \ker A\trans$. Therefore, $\ker AA^{\dagger} = \ker A\trans = (\ran A)^{\perp}$. Moreover,
  \begin{align*}
    AA^{\dagger}A = A(A\trans A)^{-1} A\trans A = A,
  \end{align*}
  which implies that the $\ran A$ is invariant under $AA^{\dagger}$. Consequently $AA^{\dagger}$ is an orthogonal projection on $\ran A$.
\end{proof}

\bibliographystyle{alphaurl}
\bibliography{bibfile}{}

\end{document}